\documentclass[a4paper,reqno,11pt]{amsart}
\usepackage{mathrsfs}
\usepackage[all]{xy}
\usepackage{txfonts}
\usepackage{amsfonts}
\usepackage{amssymb}
\usepackage{amsfonts,amsmath,amsthm,amssymb,stmaryrd,color}
\usepackage[numbers,sort&compress]{natbib}
\usepackage{hyperref}

\newtheorem{Theorem}{Theorem}[section]
\newtheorem{Lemma}{Lemma}[section]
\newtheorem{Proposition}{Proposition}[section]
\newtheorem{Remark}{Remark}[section]

\numberwithin{equation}{section}
\usepackage[left=1 in, right=1 in,top=1 in, bottom=1 in]{geometry}

\def\XXint#1#2#3{{\setbox0=\hbox{$#1{#2#3}{\int}$ }
\vcenter{\hbox{$#2#3$ }}\kern-.6\wd0}}

\DeclareMathOperator{\diver}{div}

\def\r3{\mathbb{R}^3}

\makeatletter
\@namedef{subjclassname@2020}{\textup{2020} Mathematics Subject Classification}
\makeatother

\allowdisplaybreaks

\begin{document}

\title[ The 3D inhomogeneous incompressible viscoelastic system ]{\bf G\MakeLowercase{lobal solutions of the 3}D \MakeLowercase{inhomogeneous incompressible viscoelastic system without structure assumptions}}
\author{Chengfei Ai}
\address{School of Mathematics and Statistics, Yunnan University,  Kunming, Yunnan 650091, China.}
\email[C.F. Ai]{aicf5206@163.com}

\author{Mengxing Bei}
\address{School of Mathematics and Statistics, Yunnan University,  Kunming, Yunnan 650091, China.}
\email[M. Bei]{hem135428@163.com}

\author{Yong Wang$^\ast$}
\address{ School of Mathematical Sciences, South China Normal University, Guangzhou, Guangdong 510631, China.}
\email[Y. Wang]{wangyongxmu@163.com}

\thanks{$^\ast$Corresponding author: Yong Wang,\ wangyongxmu@163.com}

\begin{abstract}
In this paper, we prove the global existence of strong solutions for the inhomogeneous incompressible viscoelastic system without any additional structure assumptions on $\mathbb{R}^{3}$. Unlike the time weighted energy method presented by Ai and Wang (Nonlinear Anal. 254 (2025), 113747.), by replacing $H^{-1}$ conditions with certain $L^{1}$ conditions on initial data, we need to develop some new transformation techniques for the system \eqref{1.1} and make use of elegant spectral analysis method to capture an enhanced time-decay rate of the velocity field $u$, which is essential to establish the uniform bounds of the density and deformation tensor.
\end{abstract}


\maketitle

\section{Introduction}
In this paper, we consider the Cauchy problem of the three-dimensional (3D) inhomogeneous incompressible viscoelastic system:
\begin{equation}\label{1.1}
\begin{cases}
\tilde{\rho}_{t}+u\cdot\nabla\tilde{\rho}=0,\\
\tilde{\rho} u_{t}+\tilde{\rho}u\cdot\nabla u+\nabla p=\mu\Delta u+c^{2}\diver(\tilde{\rho} \mathbb{F}\mathbb{F}^{\top}),\\
\mathbb{F}_{t}+ u\cdot\nabla \mathbb{F}=\nabla u\mathbb{F}, \\
\diver u=0, \qquad\qquad (x,t)\in\mathbb{R}^{3}\times\mathbb{R}^{+},
\end{cases}
\end{equation}
which is subject to the initial data
\begin{equation}\label{1.1'}
(\tilde{\rho},u,\mathbb{F})(x,t)\mid_{t=0}=(\tilde{\rho}_{0}(x),u_{0}(x),\mathbb{F}_{0}(x))\to (1,0,\mathbb{I})\quad \mbox{as}\quad x\to\infty.
\end{equation}
Here $\tilde{\rho}>0$ is the fluid density, $u\in\mathbb{R}^{3}$ is the velocity field, $\mathbb{F}\in \mathbb{M}^{3\times3}$ (the set of $3\times 3$ matrices with positive determinants) is the deformation gradient of fluids, $p$ is the pressure (the Lagrange multiplier), and $\mathbb{I}$ is the identity matrix. The constants $\mu>0$ and $c>0$ represent the shear viscosity and the speed of elastic wave propagation, respectively. The superscript $\top$ represents the transpose.

As an important part of the classical non-Newtonian fluids with elastic property. Viscoelastic fluids exhibit a combination of both fluid and solid characteristics, due to the physical importance and mathematical challenges, which have received a great deal of interest. When the density is constant (namely,  $\tilde{\rho}\equiv1$) in (\ref{1.1}), the following the incompressible version of system (\ref{1.1}) first established by Lin et al. \cite{Lin-Liu-Zhang2005} by an energetic variational approach (cf. \cite{Eisenberg-Hyon-Liu2010,Giga-Kirshtein-Liu2018,Xu-Sheng-Liu2014}) :
\begin{equation}\label{01.1}
\begin{cases}
u_{t}+u\cdot\nabla u+\nabla p=\mu\Delta u+c^{2}\diver(\mathbb{F}\mathbb{F}^{\top}),\\
\mathbb{F}_{t}+ u\cdot\nabla \mathbb{F}=\nabla u\mathbb{F}, \\
\diver u=0, \qquad\qquad (x,t)\in\mathbb{R}^{n}\times\mathbb{R}^{+},~~(n=2,3).
\end{cases}
\end{equation}
Let $\mathbb{E}=\mathbb{F}-\mathbb{I}$ be the perturbation for the deformation tensor $\mathbb{F}$ in (\ref{01.1}) with respect to the identity matrix $\mathbb{I}$ as its equilibrium state. Lin et al. \cite{Lin-Liu-Zhang2005} considered an auxiliary vector field with the ``div" structure $\diver\mathbb{E}^{\top}=0$, then they first proved the local and global existence (with small initial data) of classical solutions to the two-dimensional case. Later, Lei et al. \cite{Lei-Liu-Zhou2008} investigated the following ``curl" structure
\begin{equation}\label{01.4}
\nabla_{k}\mathbb{E}^{ij}-\nabla_{j}\mathbb{E}^{ik}=\mathbb{E}^{lj}\nabla_{l}\mathbb{E}^{ik}-\mathbb{E}^{lk}\nabla_{l}\mathbb{E}^{ij},
\end{equation}
which implies $\nabla\times \mathbb{E}$ is a high order term.
Then they proved the existence of global small solutions in 2 and 3 space dimensions via the ``div" structure and the ``curl" structure. Motivated by Sideris and Thomases \cite{Sideris-Thomases2005}, Chen and Zhang \cite{Chen-Zhang2006} established another ``curl-free" structure $\nabla\times(\mathbb{F}^{-1}-\mathbb{I})=0$ to consider the same problem. Based on the above ``div-curl" structure conditions, Lin and Zhang in \cite{Lin-Zhang2008} overcame the difficulty of undefined boundary condition for the deformation tensor $\mathbb{F}$, then they considered the initial-boundary value problems in bounded domain. Later, Jiang et al. \cite{Jiang-Jiang-Wu2017} investigated the stabilizing effect of elasticity in the Rayleigh-Taylor (RT) problem of stratified immiscible viscoelastic fluids. Moreover, the free boundary problems were also considered in \cite{DiIorio-Marcati-Spirito2020,Huang-Yao-You2025}. For more related mathematical research work, readers can refer to \cite{Hu-Lin2016,Hu-Wu2015,Jiang-Jiang2021,Lai-Lin-Wang2017,Lei-Liu-Zhou2007,Lei-Zhou2005}  in the Sobolev framework and \cite{He-Zi2021,Zhang-Fang2012,Fang-Zhang-Zi2018} in the Besov framework.

For the compressible  case corresponding to (\ref{1.1}), the analysis becomes much more complicated. Let's first briefly review the following three conditions are essential to the proof of the global existence results:
\begin{equation}\label{0.3}
\tilde{\rho}\det \mathbb{F}=1;
\end{equation}
\begin{equation}\label{0.4}
\diver (\tilde{\rho}\mathbb{F}^{\top})=0;
\end{equation}
\begin{equation}\label{0.5}
\mathbb{F}^{lk}\nabla_{l}\mathbb{F}^{ij}-\mathbb{F}^{lj}\nabla_{l}\mathbb{F}^{ik}=0.
\end{equation}
Here (\ref{0.4}) and (\ref{0.5}) are also called ``div-curl" structure for the compressible viscoelastic fluids. Utilizing the conditions (\ref{0.3})--(\ref{0.5}), Qian and Zhang \cite{Qian-Zhang2010} proved the local large and global small solutions of the Cauchy problem in critical Besov spaces (a functional space in which the system is scaling invariant); see also \cite{Hu-Wang2011,Han-Zi2020}. In the meanwhile, Hu and Wang \cite{Hu-Wang2010} first proved the local existence of strong solutions with large initial data , and then Hu-Wu \cite{Hu-Wu2013} and Fu et al. \cite{Fu-Huang-Wang2023} the global existence of and the long-time behavior of strong solutions to the Cauchy problem of the compressible viscoelastic fluids with the initial perturbation in the Sobolev framework. When studying the initial boundary value problem for the  compressible viscoelastic systems, the main difficulty arises since there is no boundary condition for the deformation tensor $\mathbb{F}$. To overcome it, Qian \cite{Qian2011} introduced a physically reasonable function $\psi=(\psi^{1},\cdot\cdot\cdot,\psi^{d})$ such that $\mathbb{F}^{ij}=\nabla_{j}\psi^{i}$, which satisfies
\begin{equation*}
\psi_{t}+v=-v\cdot\nabla\psi,
\end{equation*}
and $\psi|_{\partial\Omega}=0,$ since $v|_{\partial\Omega}=0$. Combining the tangential derivatives estimation with the normal derivatives estimation, the authors investigated the initial boundary value problem for the  compressible viscoelastic systems. Subsequently, by virtue of the theory of the maximal regularities, Hu et al. \cite{Hu-Wang2015} considered the same problem and Chen et al. \cite{Chen-Wu2018} proved the exponential decay rates of strong solutions in the $H^{2}$-Sobolev space. When studying the zero shear viscosity case, Hu and Zhao \cite{Hu-zhao2020,Hu-zhao020} justified that the strictly positive volume viscosity restrained the shock formation by establishing the global existence of the classical solutions. As mentioned in \cite{Bresch-Jabin2018}, it is important to consider the compressible viscoelastic system with the general pressure law, including the case $p'(\bar{\rho})\leq0$. Liu et al. \cite{Liu-Meng-Wu-Zhang2023} considered studied the compressible viscoelastic flows with the general pressure law and the density-dependent viscosity coefficients, then the authors proved the global existence and uniqueness of the strong solution in the critical Besov spaces. Hu et al. \cite{Hu-meng-Zhang2025} studied the three dimensional compressible viscoelastic flows with zero shear viscosity, and a general class of pressure laws. By reformulating the systems with the new perturbation variables $(\rho-1,u,\mathbb{F}-\frac{1}{\rho})$ and $(\rho-1,u,\mathbb{F}-\mathbb{I})$ to deal with the compressible and incompressible parts separately and using the elegant weighted energy method, then the authors established global existence of classical solutions for the Cauchy problem. For more related results with regard to the compressible viscoelastic model, please refer to \cite{Gu-Wang-Xie2024,Hu-Ou-Wang-Yang2023,Pan-Xu-Zhu2022,Tan-Wang-Wu2020,Zhu2020,Wu-Wang2023} and the literature therein. 

Recently, by using similar processing methods for homogeneous incompressible and compressible cases, under the identities (\ref{0.3})--(\ref{0.5}), some remarkable results can be found in \cite{Fang-Han-Zhang2014,Han2016,Hu-Wang2009} for the inhomogeneous incompressible viscoelastic system \eqref{1.1}. However, without using the ``div-curl" structural identities \eqref{0.4}-\eqref{0.5}, it is also possible to prove the global well-posedness of the Cauchy problem for the inhomogeneous incompressible viscoelastic system \eqref{1.1}, which is a new and difficult problem. Significantly, by treating the wildest ``nonlinear term" as ``linear term" through an elegant time-weighted energy framework, Zhu et al. \cite{Chen-Zhu2023,Zhu2018,Zhu2022,Pan-Xu-Zhu2022} obtained the global existence of small solutions to the homogeneous incompressible and compressible viscoelastic system in Sobolev spaces and in critical Besov spaces respectively. Later, we extended the previous results to the 3D and 2D inhomogeneous incompressible system \eqref{1.1} by introducing a new effective tensor to transform the original system into a suitable dissipative one in \cite{Ai-Wang2025,Ai-Wang-Wu2026}. As pointed out in the Remark 1.2 of \cite{Chen-Zhu2023}, if the initial data only belongs to the non-negative index Sobolev space, the existence of global classical solutions remains unknown. Fortunately, Zhu et al. \cite{Huang-Liu-Zhu2026} employed spectral analysis techniques to achieve an enhanced time-decay rate of the velocity field $u$ under certain $L^{1}$ conditions rather than $H^{-1}$ conditions on initial data in \cite{Zhu2022}. Then they proved the global strong solutions of the Cauchy problem for the 3D compressible viscoelastic equations without any additional structure assumptions. Inspired by the new methods in \cite{Huang-Liu-Zhu2026}, we continue to consider the incompressible inhomogeneous viscoelastic system (\ref{1.1}) without any additional structure assumptions. But our results are not completely parallel to the results in \cite{Huang-Liu-Zhu2026}. The main difficulties are as follows: since the density in the system \eqref{1.1} is not constant at all, the velocity equation in \eqref{1.1} is a parabolic equation with the variable coefficient, and thus the pressure term $\nabla p$ cannot be directly eliminated through the Helmholtz projection operator. In order to overcome these difficulties mentioned above, we need to develop some new transformation techniques for the system \eqref{1.1} and make use of spectral analysis methods. Specifically, the key point is to introduce the effective tensor $G$ defined in \eqref{2.1}, then we can transform the system (\ref{1.1}) into a suitable dissipative system (\ref{2.4}). To handle the pressure term $\nabla p$, another key idea is the transformation of \eqref{3.6}. Utilizing spectral analysis techniques to capture an enhanced time-decay rate of the velocity field $u$, we can obtain the uniform bounds for the Sobolev norms of $\tilde{\rho}-1$ and $\mathbb{F}-\mathbb{I}$. Lastly, we extend the novel ideas in \cite{Huang-Liu-Zhu2026} into the inhomogeneous incompressible case, and prove the existence of global small solutions for the Cauchy problem without any additional structure assumptions on $\mathbb{F}$ in Sobolev space.

\medskip

\textbf{Notation.} In this paper, $C$ denotes a general constant may vary in different estimate. If the dependence need to be explicitly stressed, some notations like $C_0,C_1$ will be used. $a\lesssim b ~(a\gtrsim b)$ means that both $a\leq C b~(a\geq C b)$ for some constant $C>0$, while $a\sim b$ means that both $a\lesssim b$ and $a\gtrsim b$. Let $\nabla^{k}=\partial_{x}^{k}$ with an integer $k\geq0$ be the usual spatial derivatives of order $k$. Moreover, for $s<0$ or $s$ is not a positive integer, $\nabla^{s}$ stands for $\Lambda^{s}$, that is,
\begin{equation*}
\nabla^{s}f=\Lambda^{s}f:=\mathscr{F}^{-1}(|\xi|^s\mathscr{F}f),
\end{equation*}
where $\mathscr{F}$ is the usual Fourier transform operator and $\mathscr{F}^{-1}$ is its inverse (see e.g., \cite{Bahouri-Chemin-Danchin2011}). We use $\dot{H}^s(\mathbb{R}^n)$ $(s\in\mathbb{R})$ to denote the homogeneous Sobolev spaces on $\mathbb{R}^n$ with the norm $\|\cdot\|_{\dot{H}^s}$ defined by $\|f\|_{\dot{H}^s}:=\|\Lambda^sf\|_{L^2}$. And $H^s(\mathbb{R}^n)$ and $L^{p}(\mathbb{R}^n)$ are the usual Sobolev and Lebesgue spaces with the norm $\|\cdot\|_{H^s}$ and the norm $\|\cdot\|_{L^{p}}$, respectively. Throughout the paper, $\left\langle f,g\right\rangle_{X}$ denotes the $X-$inner product of $f$ and $g$.

For simplicity, we shall always use the Einstein convention of summation, and we do not distinguish functional spaces when scalar-valued or vector-valued functions are involved.

Now we are in a position to present the main result.

\begin{Theorem}\label{th1.1}
Suppose that the initial data $(\tilde{\rho}_{0}, u_{0}, \mathbb{F}_{0})$ with $\diver u_0=0$ satisfies for some sufficiently small constant $\delta>0$,
\begin{equation}\label{0.7}
\|u_{0}\|_{H^{2}}+\|\tilde{\rho}_{0}-1\|_{H^{2}}+\|G_{0}\|_{H^{2}}\leq\delta,
\end{equation}
and
\begin{equation}\label{0.8}
\|u_{0}\|_{L^{1}}+\|G_{0}\|_{L^{1}}=M_{0}<+\infty.
\end{equation}
 Then the Cauchy problem \eqref{1.1}--\eqref{1.1'} admits a unique global solution $(\tilde{\rho}, u, \mathbb{F})(t)$ such that
\begin{align}\label{0.9}
&\sup_{0\leq t\leq \infty}\big[\|(\tilde{\rho}-1,\mathbb{F}-\mathbb{I})(t)\|_{H^{2}}^{2}+(1+t)^{\sigma}\|u(t)\|_{H^{2}}^{2}\big]
+\int_{0}^{\infty}(1+t)^{\sigma}\|\nabla u(t)\|_{H^{2}}^{2}dt \nonumber \\
&\lesssim\|(u_{0},\tilde{\rho}_{0}-1,\mathbb{F}_{0}-\mathbb{I})\|_{H^{2}}^{2}+M_{0}^{2},
\end{align}
where $G:=\tilde{\rho}\mathbb{F}\mathbb{F}^{T}-\mathbb{I}, \sigma\in(1,\sigma_{0})$ and $1<\sigma_{0}<\frac{3}{2}$.
\end{Theorem}

We give some remarks in the following.

\begin{Remark}
In this paper, we relax the previous smallness assumption conditions on initial data in \cite{Ai-Wang2025} by replacing $H^{-1}$ conditions with certain $L^{1}$ conditions. Moreover, the constant $M_{0}$ in Theorem 1.1 is not required to be small.
\end{Remark}

\begin{Remark}
As pointed out in \cite{Huang-Liu-Zhu2026}, the precise selection of $\sigma$ is crucial, which is used to capture the time-decay of $u$ and obtain the uniform bounds of $\rho$ and $\mathbb{F}-\mathbb{I}$. In our results, we obtained that $\sigma\in(1,\sigma_{0})$ and $1<\sigma_{0}<\frac{3}{2}$, and its range of values is optimal.
\end{Remark}

The rest of this paper is structured as follows. In Section \ref{se2}, we transform system (\ref{1.1}) into a suitable dissipative system for the velocity $u$ and the effective flux $\mathcal{G}$ and present several useful lemmas which are frequently used in the following sections. In Section \ref{se3}, with the help of the previous results, we establish some key priori estimates by utilizing standard energy method and spectral method, and then finish the proof of Theorem 1.1.

\section{Derivation of the dissipative system}\label{se2}
In this section, we first make appropriate transformations to system (\ref{1.1}). Then we give some notations and useful results that play important roles in the subsequent sections.
\subsection{Reformulation}
\ \ \ \\

Since the specific values of the positive coefficients $\mu>0, c>0$ in (\ref{1.1}) are not essential in the present article, so we take $\mu=c=1$ and define $\rho:=\tilde{\rho}-1$ in the rest of this paper. Next, in order to effectively analyze (\ref{1.1}), we provide the definition of the effective tensor $G$ as follows
\begin{equation}\label{2.1}
G:=\tilde{\rho}\mathbb{F}\mathbb{F}^{\top}-\mathbb{I}.
\end{equation}
And we have
\begin{equation}\label{2.2}
(\tilde{\rho}\mathbb{F}\mathbb{F}^{\top})_{t}+(u\cdot\nabla\tilde{\rho})\mathbb{F}\mathbb{F}^{\top}+\tilde{\rho}(u\cdot\nabla\mathbb{F})\mathbb{F}^{\top}
+\tilde{\rho}\mathbb{F}(u\cdot\nabla\mathbb{F}^{\top})=\tilde{\rho}(\nabla u \mathbb{F})\mathbb{F}^{\top}+\tilde{\rho}\mathbb{F}\mathbb{F}^{\top}(\nabla u)^{\top}.
\end{equation}
From (\ref{2.1}) and (\ref{2.2}), we can derive the evolution for effective tensor $G$
\begin{equation}\label{2.3}
G_{t}+u\cdot \nabla G+Q(\nabla u, G)=2D(u),
\end{equation}
where $Q(\nabla u, G):=-\nabla uG-G(\nabla u)^{\top}$ and $D(u):=\frac{1}{2}(\nabla u+(\nabla u)^{\top})$.

Combining (\ref{1.1}) with (\ref{2.3}), the following new system is established
\begin{equation}\label{2.4}
\begin{cases}
u_{t}-\Delta u+u\cdot\nabla u+(1-\frac{1}{\rho+1})\Delta u+\frac{1}{\rho+1}\nabla p+(1-\frac{1}{\rho+1})\diver G=\diver G,\\
G_{t}+ u\cdot\nabla G+Q(\nabla u, G)=2D(u), \\
\diver u=0, \qquad\qquad (x,t)\in\mathbb{R}^{3}\times\mathbb{R}^{+}.
\end{cases}
\end{equation}
Moreover, we introduce the effective flux $\mathcal{G}:=\Lambda^{-1}\diver G$, which seems more compatible with $u$. Then the following evolution of $(u,\mathcal{G})$ can be derived from (\ref{2.4})
\begin{equation}\label{2.5}
\begin{cases}
u_{t}-\Delta u+u\cdot\nabla u+(1-\frac{1}{\rho+1})\Delta u+\frac{1}{\rho+1}\nabla p+(1-\frac{1}{\rho+1})\Lambda\mathcal{G}=\Lambda\mathcal{G},\\
\mathcal{G}_{t}+\Lambda^{-1}\diver (u\cdot\nabla G)+\Lambda^{-1}\diver Q(\nabla u, G)=-\Lambda u, \\
\diver u=0, \qquad\qquad (x,t)\in\mathbb{R}^{3}\times\mathbb{R}^{+}.
\end{cases}
\end{equation}
\subsection{Some useful lemmas}
\ \ \ \\

In this subsection, we present some useful lemmas which are frequently used in later sections. First,
we recall the general Gagliardo-Nirenberg inequality.

\begin{Lemma}\label{le2.1}
Let $1\leq q, r\leq \infty$ and $0\leq m\leq\alpha< l$, then we have
\begin{equation}\label{2.6}
\|\nabla^{\alpha}f\|_{L^p}\lesssim \|\nabla^{m}f\|_{L^q}^{1-\theta}\|\nabla^{l}f\|_{L^r}^{\theta},
\end{equation}
where $\frac{\alpha}{l}\leq\theta\leq1$ and $\alpha$ satisfies
\begin{equation*}
  \frac{\alpha}{3}-\frac1p=\left(\frac{m}{3}-\frac{1}{q}\right)(1-\theta)+\left(\frac{l}{3}-\frac{1}{r}\right)\theta
\end{equation*}
with the following exceptional cases: \\
$(1)$ If $\alpha=0, rl<3, q=\infty$, then one needs the additional assumption that either $f$ tends to zero at infinity or $f\in L^{\tilde{q}}$ for some finite $\tilde{q}>0$;\\
$(2)$ If $1<r<\infty$, and $l-\alpha-\frac{3}{r}$ is a non negative integer, then the interpolation inequalities in (\ref{2.6}) hold only for $\theta$ satisfying $\frac{\alpha}{l}\leq \theta<1$;\\
$(3)$ The interpolation inequalities in (\ref{2.6}) also hold for fractional derivatives.
\end{Lemma}

\begin{proof}
See Theorem (pp. 125--126) in \cite{Nirenberg1959}.
\end{proof}

\begin{Lemma}\label{le2.2}
If any smooth function $g(\cdot)$ is defined around $0$ with $g(0)=0$, which satisfies
\begin{equation*}
\text{$g(\rho)\sim\rho$\ \ \  and  \ \ \  $\|g^{(k)}(\rho)\|_{L^{2}}\leq C(k)$\ \ \ \ \ \  for any \ \  $0\leq k \leq2$ },
\end{equation*}
then it holds that
\begin{align*}
&\|g(\rho)\|_{L^{p}}\lesssim\|\rho\|_{L^{p}}, \ \ \ \text{for some~$p$ with $1\leq p\leq\infty$} , \\
&\|\nabla^{k}g(\rho)\|_{L^{p}}\lesssim\|\nabla^{k}\rho\|_{L^{p}}, \ \ \ k=1,2.
\end{align*}
\end{Lemma}
\begin{proof}
See Proposition 2.2 in \cite{Zhu2022}.
\end{proof}

\begin{Lemma}\label{le2.3}
Assume that $\alpha\geq0, \beta>0$ and $c>0$ are real numbers. Let $1\leq p\leq q\leq\infty$. Then there exists a constant $C>0$ such that
\begin{equation}\label{2.7}
\|\Lambda^{\alpha}e^{-c\Lambda^{\beta}t}f\|_{L^{q}(\mathbb{R}^{3})}\leq C t^{-\frac{\alpha}{\beta}-\frac{3}{\beta}\Big(\frac{1}{p}-\frac{1}{q}\Big)}\|f\|_{L^{p}(\mathbb{R}^{3})}
\end{equation}
for any $t>0$.
\end{Lemma}
\begin{proof}
See Lemma 2.3 in \cite{Dong-Li-Wu2019}.
\end{proof}

\begin{Lemma}\label{le2.4}
If $0<s_{1}\leq s_{2}$, then there exists a constant $C>0$ such that
\begin{equation}\label{2.8}
\int_{0}^{t}(1+t-\tau)^{-s_{1}}(1+\tau)^{-s_{2}}d\tau\leq
\begin{cases}
C(1+t)^{-s_{1}}, & \text{if $s_{2}>1$},\\
C(1+t)^{-s_{1}}\ln(1+t), & \text{if $s_{2}=1$},\\
C(1+t)^{1-s_{1}-s_{2}}, & \text{if $s_{2}<1$}.
\end{cases}
\end{equation}
\end{Lemma}
\begin{proof}
See Lemma 2.4 in \cite{Wan2019}.
\end{proof}

If $s\in [0,\frac{3}{2})$, we can infer from the Hardy-Littlewood-Sobolev theorem that the following $L^p$-type inequality holds:
\begin{Lemma}\label{le2.5}
Let $0\leq s<\frac{3}{2}, 1<p\leq2$ with $\frac{1}{2}+\frac{s}{3}=\frac{1}{p}$, it holds that
\begin{equation*}
\|f\|_{\dot{H}^{-s}}\lesssim\|f\|_{L^{p}}.
\end{equation*}
\end{Lemma}
\begin{proof}
See Theorem 1 (p. 119) in \cite{Stein1970}.
\end{proof}

Finally, we also provide the following useful regularity results for the Stokes problem:
\begin{Lemma}\label{le2.6}
Assume that $f\in L^{r}(\mathbb{R}^{n})$ with $2\leq r<\infty$. Let $(u,p)\in H^{1}(\mathbb{R}^{n})\times L^{2}(\mathbb{R}^{n})$ be the unique weak solution to the following Stokes problem
\begin{equation*}
\begin{cases}
-\mu\Delta u+\nabla p=f,\\
\diver u=0,\\
u(x)\rightarrow0, & x\rightarrow\infty.
\end{cases}
\end{equation*}
Then $(\nabla^{2}u,\nabla  p)\in L^{r}(\mathbb{R}^{n})$ and satisfies
\begin{equation*}
\|\nabla^{2}u\|_{L^{r}(\mathbb{R}^{n})}+\|\nabla p\|_{L^{r}(\mathbb{R}^{n})}
\lesssim \|f\|_{L^{r}(\mathbb{R}^{n})}.
\end{equation*}
\end{Lemma}
\begin{proof}
See Lemma 4.3 (p. 322) in \cite{Galdi2011}.
\end{proof}

\section{Proof of Theorem 1.1 }\label{se3}
In this section, we mainly focus on proving Theorem 1.1. We first define the solution space $E$ by
\begin{equation*}
E:=\left\{(\tilde{\rho}-1,\mathbb{F}-\mathbb{I})\in L^{\infty}(0,T; H^{2}\times H^{2}), u\in C\big([0,T); H^{2}\big), u\in L^{2}\big((0,T); H^{2}\big)\right\}.
\end{equation*}
Since the local existence and uniqueness of the solution to the Cauchy problem \eqref{1.1}--\eqref{1.1'}, which can be established by the standard method in \cite{Chemin-Masmoudi2001}. In order to prove Theorem 1.1, the main effort is devoted to establishing a priori estimate stated in the following proposition.
\begin{Proposition}\label{p3.1}
Assume $(\tilde{\rho}, u, \mathbb{F})\in E$ is a solution of the Cauchy problem \eqref{1.1}--\eqref{1.1'} for some $T>0$. Under conditons (\ref{0.7}), (\ref{0.8}), $(\tilde{\rho}, u, \mathbb{F})$ satisfies for some sufficiently small constant $\delta>0$,
\begin{equation}\label{3.1}
\sup_{0\leq t\leq T}\|(\rho,u,\mathbb{F}-\mathbb{I})\|_{H^{2}}\leq\delta,
\end{equation}
Then the following uniform estimates are established:
\begin{equation}\label{3.2}
\|(\rho,u,\mathbb{F}-\mathbb{I})\|_{H^{2}}\lesssim\|(\rho_{0},u_{0},\mathbb{F}_{0}-\mathbb{I})\|_{H^{2}}+\|(\rho_{0},u_{0},\mathbb{F}_{0}-\mathbb{I})\|_{H^{2}}^{\frac{\sigma-1}{2\sigma}}M_{0}^{\frac{\sigma+1}{2\sigma}},
\end{equation}
and
\begin{align}\label{3.3}
&\|(\rho,\mathbb{F}-\mathbb{I})(t)\|_{H^{2}}^{2}+(1+t)^{\sigma}\|u(t)\|_{H^{2}}^{2}
+\int_{0}^{t}(1+\tau)^{\sigma}\|\nabla u(\tau)\|_{H^{2}}^{2}d\tau \nonumber \\
&\lesssim\|(u_{0},\rho_{0},\mathbb{F}_{0}-\mathbb{I})\|_{H^{2}}^{2}+M_{0}^{2},
\end{align}
where $\sigma\in(1,\sigma_{0})$ and $1<\sigma_{0}<\frac{3}{2}$.
\end{Proposition}

Proposition \ref{p3.1} will be justified by several lemmas in the rest of this section. Before proceeding any further, we assume a
priori estimate that
\begin{equation}\label{3.4}
\|(\rho,u,G)\|_{H^{2}}\leq\delta\ll1,
\end{equation}
where $\delta$ satisfies $C_{0}\delta<\sigma_{0}-\sigma$, $C_{0}$ is a positive constant independent of $\delta$ , $\sigma$ and $\sigma_{0}$, which can be completely determined later by $(\ref{3.8})$, $(\ref{3.39})$ and $(\ref{3.40})$. Moreover, from (\ref{3.4}) and the definitions of $\mathcal{G}$ and $G$, we have
\begin{equation}\label{3.5}
\|\mathcal{G}\|_{H^{2}}\lesssim\|G\|_{H^{2}}\lesssim\delta.
\end{equation}

Next, we first establish the energy estimates of $(u,\mathcal{G})$. The main difficulty is to deal with the pressure term $\frac{1}{\rho+1}\nabla p$ in the first equation of system (\ref{2.5}). Through the method in \cite{Ai-Wang-Wu2026}, we can transform the pressure $p$ to $\tilde{p}$ by
\begin{equation}\label{3.6}
\nabla\tilde{p}:=\frac{1}{\rho+1}\nabla p.
\end{equation}
Considering (\ref{3.6}), the system $(\ref{2.5})$ can be written as
\begin{equation}\label{3.7}
\begin{cases}
u_{t}-\Delta u+u\cdot\nabla u+(1-\frac{1}{\rho+1})\Delta u+\nabla\tilde{p}+(1-\frac{1}{\rho+1})\Lambda\mathcal{G}=\Lambda\mathcal{G},\\
\mathcal{G}_{t}+\Lambda^{-1}\diver (u\cdot\nabla G)+\Lambda^{-1}\diver Q(\nabla u, G)=-\Lambda u, \\
\diver u=0, \qquad\qquad (x,t)\in\mathbb{R}^{3}\times\mathbb{R}^{+}.
\end{cases}
\end{equation}
The energy estimates of $(u,\mathcal{G})$ can be established in the following lemma.
\begin{Lemma}\label{le3.1}
Under the assumptions in Proposition \ref{p3.1}, then there exists a sufficiently small constant $\delta>0$ such that
\begin{equation}\label{3.8}
\sup_{0\leq t'\leq t}\|(u,\mathcal{G})(t')\|_{H^{2}}^{2}+\int_{0}^{t}\left(\|\Lambda u(\tau)\|_{H^{2}}^{2}+\|\Lambda \mathcal{G}(\tau)\|_{H^{1}}^{2}\right)d\tau
\lesssim\|(\rho_{0},u_{0},\mathbb{F}_{0}-\mathbb{I})\|_{H^{2}}^{2}.
\end{equation}
\end{Lemma}
\begin{proof}
Applying the operators $\Lambda^{k}\mathbb{P}, \Lambda^{k}, (k=0,1,2)$ to $(\ref{3.7})_{1}, (\ref{3.7})_{2}$, respectively. We get
\begin{equation}\label{3.9}
\begin{cases}
\Lambda^{k}u_{t}-\Lambda^{k}\Delta u+\Lambda^{k}\mathbb{P}(u\cdot\nabla u)+\Lambda^{k}\mathbb{P}\big[\frac{\rho}{\rho+1}(\Delta u+\diver  G)\big]=\Lambda^{k+1}\mathbb{P} \mathcal{G},\\
\Lambda^{k}\mathcal{G}_{t}+\Lambda^{k-1}\diver (u\cdot\nabla G)+\Lambda^{k-1}\diver Q(\nabla u, G)=-\Lambda^{k+1} u.
\end{cases}
\end{equation}
Then taking inner product with $(\Lambda^{k}u,\Lambda^{k}\mathcal{G})$ for $(\ref{3.9})_{1}, (\ref{3.9})_{2}$, respectively. We have
\begin{align}\label{3.10}
\frac{1}{2}\frac{d}{dt}\|\left(\Lambda^{k} u,\Lambda^{k}\mathcal{G}\right)\|_{L^{2}}^{2}+\|\Lambda^{k+1} u\|_{L^{2}}^{2}=I_{1}+I_{2}+I_{3}+I_{4}+I_{5}+I_{6},
\end{align}
where
\begin{align*}
&I_{1}=\left\langle\Lambda^{k+1}\mathbb{P} \mathcal{G},\Lambda^{k}u\right\rangle_{L^{2}};\quad~~~~I_{2}=-\left\langle\Lambda^{k+1}u,\Lambda^{k}\mathcal{G}\right\rangle_{L^{2}};\\
&I_{3}=-\left\langle\Lambda^{k}\mathbb{P}(u\cdot\nabla u),\Lambda^{k}u\right\rangle_{L^{2}}; ~~
I_{4}=-\left\langle\Lambda^{k}\mathbb{P}\Big[\frac{\rho}{\rho+1}(\Delta u+\diver  G)\Big],\Lambda^{k}u\right\rangle_{L^{2}};\\
&I_{5}=-\left\langle\Lambda^{k-1}\diver Q(\nabla u, G),\Lambda^{k}\mathcal{G}\right\rangle_{L^{2}};~~I_{6}=-\left\langle\Lambda^{k-1}\diver (u\cdot\nabla G),\Lambda^{k}\mathcal{G}\right\rangle_{L^{2}}.
\end{align*}
Making use of integration by parts and considering $\mathbb{P}u=u$, one gets
\begin{equation*}
I_{1}+I_{2}=0.
\end{equation*}
For $I_{3}$, it follows from integration by parts, the divergence-free condition, H\"{o}lder's inequality, Lemma  and Sobolev imbedding theorem, we have
\begin{align}\label{3.11}
I_{3}&\lesssim\Big(\|\Lambda^{-1}(u\cdot\nabla u)\|_{L^{2}}+\|u\cdot\nabla u\|_{H^{1}}\Big)\|\Lambda u\|_{H^{2}} \nonumber \\
&\lesssim\Big(\|u\|_{L^{3}}\|\nabla u\|_{L^{2}}+\|u\|_{L^{\infty}}\|\nabla u\|_{H^{1}}+\|\nabla u\|_{L^{6}}\|\nabla u\|_{L^{3}}\Big)\|\Lambda u\|_{H^{2}} \nonumber \\
&\lesssim\delta\|\Lambda u\|_{H^{2}}^{2}.
\end{align}
Let
\begin{align*}
g(\rho):=\frac{\rho}{\rho+1}.
\end{align*}
For $I_{4}, I_{5}$, utilizing integration by parts, Sobolev imbedding theorem, H\"{o}lder
and Young's inequalities and the boundedness of Riesz operator $\mathcal{R}_{i}=\Lambda^{-1}\nabla_{i}$ from $L^{2}$ into itself, we get
\begin{align}\label{3.12}
I_{4}&\lesssim\Big(\|\Lambda^{-1}\big(g(\rho)\Delta u+g(\rho)\Lambda\mathcal{G}\big)\|_{L^{2}}+\|g(\rho)\Delta u+g(\rho)\Lambda\mathcal{G}\|_{H^{1}}\Big)\|\Lambda u\|_{H^{2}} \nonumber \\
&\lesssim\Big(\left\|g(\rho)\right\|_{W^{1,3}}+\left\|g(\rho)\right\|_{L^{\infty}}\Big)
\Big(\|\Lambda^{2} u\|_{H^{1}}+\|\Lambda\mathcal{G}\|_{H^{1}}\Big)\|\Lambda u\|_{H^{2}} \nonumber \\
&\lesssim\delta\Big(\|\Lambda u\|_{H^{2}}^{2}+\|\Lambda\mathcal{G}\|_{H^{1}}^{2}\Big),
\end{align}
and
\begin{align}\label{3.13}
I_{5}&\lesssim\|\Lambda^{-2}\diver Q(\nabla u, G)\|_{L^{2}}\|\Lambda\mathcal{G}\|_{L^{2}}+\|\diver Q(\nabla u, G)\|_{H^{1}}\|\Lambda\mathcal{G}\|_{H^{1}} \nonumber \\
&\lesssim\|G\|_{L^{3}}\|\nabla u\|_{L^{2}}\|\Lambda\mathcal{G}\|_{L^{2}}+\Big(\|G\|_{L^{\infty}}\|\Lambda^{2} u\|_{H^{1}}+\|\Lambda\mathcal{G}\|_{L^{3}}\|\Lambda^{2} u\|_{H^{1}}+\|\nabla\Lambda\mathcal{G}\|_{L^{2}}\|\nabla u\|_{L^{\infty}}\Big)\|\Lambda\mathcal{G}\|_{H^{1}} \nonumber \\
&\lesssim\|G\|_{H^{1}}\|\nabla u\|_{L^{2}}\|\Lambda\mathcal{G}\|_{L^{2}}+\Big(\|G\|_{H^{2}}\|\Lambda^{2} u\|_{H^{1}}+\|\Lambda\mathcal{G}\|_{H^{1}}\|\Lambda^{2} u\|_{H^{1}}\Big)\|\Lambda\mathcal{G}\|_{H^{1}} \nonumber \\
&\lesssim\delta\Big(\|\Lambda u\|_{H^{2}}^{2}+\|\Lambda\mathcal{G}\|_{H^{1}}^{2}\Big).
\end{align}
For $I_{6}$, similar to (\ref{3.12}) and (\ref{3.13}), using integration by parts and the divergence-free condition, we have
\begin{align}\label{3.14}
I_{6}&=-\left\langle\Lambda^{-1}\diver (u\cdot\nabla G),\mathcal{G}\right\rangle_{L^{2}}-\left\langle\diver (u\cdot\nabla G),\Lambda\mathcal{G}\right\rangle_{L^{2}}-\left\langle\Lambda(\diver (u\cdot\nabla G)),\Lambda^{2}\mathcal{G}\right\rangle_{L^{2}}\nonumber \\
&\lesssim\|G\|_{L^{3}}\|\Lambda u\|_{L^{2}}\|\Lambda\mathcal{G}\|_{L^{2}}+\|\nabla u\nabla G+u\cdot\nabla\Lambda\mathcal{G}\|_{L^{2}}\|\Lambda\mathcal{G}\|_{L^{2}}+\Big(\|\Lambda(\nabla u\nabla G)\|_{L^{2}}+\|\Lambda u\cdot\nabla\Lambda\mathcal{G}\|_{L^{2}}\Big)\|\Lambda^{2}\mathcal{G}\|_{L^{2}} \nonumber \\
&\ \ \ \ -\left\langle u\cdot\nabla\Lambda^{2}\mathcal{G},\Lambda^{2}\mathcal{G}\right\rangle_{L^{2}} \nonumber \\
&\lesssim\|G\|_{H^{1}}\|\Lambda u\|_{L^{2}}\|\Lambda\mathcal{G}\|_{L^{2}}+\Big(\|\nabla u\|_{L^{6}}\|\nabla G\|_{L^{3}}+\|u\|_{L^{\infty}}\|\nabla\Lambda\mathcal{G}\|_{L^{2}}\Big)\|\Lambda\mathcal{G}\|_{L^{2}} \nonumber \\
&\ \ \ \ +\Big(\|\Lambda\nabla u\|_{L^{6}}\|\nabla G\|_{L^{3}}+\|\nabla u\|_{L^{\infty}}\|\nabla\Lambda G\|_{L^{2}}+\|\Lambda u\|_{L^{\infty}}\|\nabla\Lambda\mathcal{G}\|_{L^{2}}\Big)\|\Lambda^{2}\mathcal{G}\|_{L^{2}} \nonumber \\
&\lesssim\delta\Big(\|\Lambda u\|_{H^{2}}^{2}+\|\Lambda\mathcal{G}\|_{H^{1}}^{2}\Big).
\end{align}
Putting all the estimates of (\ref{3.11})-(\ref{3.14}) together, and selecting $\delta$ small enough, we get
\begin{equation}\label{3.15}
\frac{d}{dt}\|\left(u,\mathcal{G}\right)\|_{H^{2}}^{2}+\|\Lambda u\|_{H^{2}}^{2}\lesssim\delta\|\Lambda\mathcal{G}\|_{H^{1}}^{2}.
\end{equation}

Next, applying the operators $\Lambda^{k}\mathbb{P}, \Lambda^{k+1}, (k=0,1)$ to $(\ref{3.7})_{1}, (\ref{3.7})_{2}$, respectively, and taking inner product of the resulting equations with $(-\Lambda^{k+1}\mathcal{G},-\Lambda^{k}u)$. We have
\begin{align}\label{3.16}
\frac{d}{dt}\left\langle\Lambda^{k}u,-\Lambda^{k+1}\mathcal{G}\right\rangle_{L^{2}}+\|\Lambda^{k+1} \mathcal{G}\|_{L^{2}}^{2}=I_{7}+I_{8}+I_{9}+I_{10}+I_{11}+I_{12},
\end{align}
where
\begin{align*}
&I_{7}=\left\langle\Lambda^{k+2}u,\Lambda^{k}u\right\rangle_{L^{2}};\quad~~~~I_{8}=\left\langle\Lambda^{k+2}u,\Lambda^{k+1}\mathcal{G}\right\rangle_{L^{2}};\\
&I_{9}=\left\langle\Lambda^{k}\mathbb{P}(u\cdot\nabla u),\Lambda^{k+1}\mathcal{G}\right\rangle_{L^{2}}; ~~
I_{10}=\left\langle\Lambda^{k}\mathbb{P}\Big[g(\rho)(\Delta u+\diver  G)\Big],\Lambda^{k+1}\mathcal{G}\right\rangle_{L^{2}};\\
&I_{11}=\left\langle\Lambda^{k}\diver Q(\nabla u, G),\Lambda^{k}u\right\rangle_{L^{2}};~~I_{12}=\left\langle\Lambda^{k}\diver (u\cdot\nabla G),\Lambda^{k}u\right\rangle_{L^{2}}.
\end{align*}
For $I_{7}, I_{8}$, using integration by parts, H\"{o}lder's and Young's inequalities, we have
\begin{align}\label{3.17}
I_{7}+I_{8}\leq C_{0}\|\Lambda u\|_{H^{1}}^{2}+C_{\delta_{0}}\|\Lambda u\|_{H^{2}}^{2}+\delta_{0}\|\Lambda\mathcal{G}\|_{H^{1}}^{2}.
\end{align}
For $I_{9}, I_{10}$, similar to (\ref{3.11}) and (\ref{3.12}), we get
\begin{align}\label{3.18}
I_{9}+I_{10}&\lesssim\Big(\|\mathbb{P}(u\cdot\nabla u)\|_{H^{1}}+\|\mathbb{P}\big[g(\rho)(\Delta u+\diver  G)\big]\|_{H^{1}}\Big)\|\Lambda\mathcal{G}\|_{H^{1}} \nonumber \\
&\lesssim\Big(\|u\|_{L^{\infty}}\|\nabla u\|_{H^{1}}+\|\nabla u\|_{L^{6}}\|\nabla u\|_{L^{3}}\Big)\|\Lambda\mathcal{G}\|_{H^{1}} \nonumber \\
&\ \ \ \ +\Big(\|g(\rho)\|_{L^{\infty}}+\|\nabla g(\rho)\|_{L^{3}}\Big)\Big(\|\Lambda^{2} u\|_{H^{1}}+\|\Lambda\mathcal{G}\|_{H^{1}}\Big)\|\Lambda\mathcal{G}\|_{H^{1}}\nonumber \\
&\lesssim\delta\Big(\|\Lambda u\|_{H^{2}}^{2}+\|\Lambda\mathcal{G}\|_{H^{1}}^{2}\Big).
\end{align}
For $I_{11}, I_{12}$, similar to (\ref{3.14}), we get
\begin{align}\label{3.19}
I_{11}+I_{12}&=\left\langle\Lambda^{k-1}\diver Q(\nabla u, G),\Lambda^{k+1}u\right\rangle_{L^{2}}+\left\langle\Lambda^{k-1}\diver (u\cdot\nabla G),\Lambda^{k+1}u\right\rangle_{L^{2}} \nonumber \\
&\lesssim\Big(\|\Lambda^{-1}\diver Q(\nabla u, G)\|_{H^{1}}+\|\Lambda^{-1}\diver (u\cdot\nabla G)\|_{H^{1}}\Big)\|\Lambda u\|_{H^{1}} \nonumber \\
&\lesssim\Big(\|G\|_{L^{\infty}}\|\nabla u\|_{H^{1}}+\|\nabla G\|_{L^{3}}\|\nabla u\|_{L^{6}}\Big)\|\Lambda u\|_{H^{1}} \nonumber \\
&\ \ \ \ +\Big(\|u\|_{L^{\infty}}\|\nabla G\|_{H^{1}}+\|\nabla G\|_{L^{3}}\|\nabla u\|_{L^{6}}\Big)\|\Lambda u\|_{H^{1}}\nonumber \\
&\lesssim\delta\Big(\|\Lambda u\|_{H^{2}}^{2}+\|\Lambda\mathcal{G}\|_{H^{1}}^{2}\Big).
\end{align}
Combining the estimates of (\ref{3.17})-(\ref{3.19}) together, and selecting $\delta, \delta_{0}$ small enough, we have
\begin{equation}\label{3.20}
\frac{d}{dt}\left\langle\Lambda^{k}u,-\Lambda^{k+1}\mathcal{G}\right\rangle_{L^{2}}+\|\Lambda^{k+1} \mathcal{G}\|_{L^{2}}^{2}\lesssim\|\Lambda u\|_{H^{2}}^{2}.
\end{equation}

Finally, multiplying (\ref{3.15}) by some suitable large constant, then adding it to (\ref{3.20}), and integrating the result from $0$ to $t$, and combining $\left\langle\Lambda^{k}u,-\Lambda^{k+1}\mathcal{G}\right\rangle_{L^{2}}\geq-\frac{1}{2}(\|u\|_{H^{1}}^{2}+\|\mathcal{G}\|_{H^{2}}^{2})$. We immediately get (\ref{3.8}). This completes the proof of Lemma \ref{le3.1}.

\end{proof}

The following lemma is about the estimates of $\rho$ and $\mathbb{F}-\mathbb{I}$.
\begin{Lemma}\label{le3.2}
Under conditons (\ref{0.7}), (\ref{0.8}), then the following estimate is given
\begin{align}\label{3.21}
\|(\rho,\mathbb{F}-\mathbb{I})(t)\|_{H^{2}}^{2}&\lesssim\sup_{0\leq \tau\leq t}\|(\rho,\mathbb{F}-\mathbb{I})(\tau)\|_{H^{2}}^{2}\int_{0}^{t}\|\nabla u(\tau)\|_{H^{2}}d\tau+\left(\int_{0}^{t}\|\nabla u(\tau)\|_{H^{2}}d\tau\right)^{2}\nonumber \\
&\ \ \ \ +\|(\rho_{0},\mathbb{F}_{0}-\mathbb{I})\|_{H^{2}}^{2}
\end{align}
for any $t>0$.
\end{Lemma}
\begin{proof}
By virtue of the first equation of system (\ref{1.1}), using integration by parts and the divergence-free condition, H\"{o}lder and Young's inequalities, we have
\begin{align}\label{3.22}
\frac{1}{2}\frac{d}{dt}\|\rho\|_{H^{2}}^{2}&=-\left\langle u\cdot\nabla\rho,\rho\right\rangle_{H^{2}}\nonumber \\
&\lesssim\|\nabla u\cdot\nabla\rho\|_{L^{2}}\|\nabla\rho\|_{L^{2}}+\|\nabla^{2} u\cdot\nabla\rho+\nabla u\cdot\nabla\nabla\rho\|_{L^{2}}\|\nabla^{2}\rho\|_{L^{2}}\nonumber \\
&\lesssim\|\nabla u\|_{H^{2}}\|\rho\|_{H^{2}}^{2}.
\end{align}
Integrating (\ref{3.22}) from $0$ to $t$, we arrive at
\begin{align}\label{3.23}
\|\rho(t)\|_{H^{2}}^{2}&\lesssim\sup_{0\leq \tau\leq t}\|\rho(\tau)\|_{H^{2}}^{2}\int_{0}^{t}\|\nabla u(\tau)\|_{H^{2}}d\tau+\|\rho_{0}\|_{H^{2}}^{2}.
\end{align}
By virtue of the third equation of system (\ref{1.1}), similar to (\ref{3.22})-(\ref{3.23}), considering the divergence-free condition, we have
\begin{align}\label{3.24}
\frac{1}{2}\frac{d}{dt}\|\mathbb{F}-\mathbb{I}\|_{H^{2}}^{2}&=-\left\langle u\cdot\nabla(\mathbb{F}-\mathbb{I}),\mathbb{F}-\mathbb{I}\right\rangle_{H^{2}}+\left\langle \nabla u\mathbb{F},\mathbb{F}-\mathbb{I}\right\rangle_{H^{2}}\nonumber \\
&\lesssim\|\nabla u\|_{L^{\infty}}\Big(\|\nabla(\mathbb{F}-\mathbb{I})\|_{L^{2}}^{2}+\|\nabla^{2}(\mathbb{F}-\mathbb{I})\|_{L^{2}}^{2}\Big)\nonumber \\
&\ \ \ \ +\Big(\|\nabla u(\mathbb{F}-\mathbb{I})\|_{H^{2}}+\|\nabla u\|_{H^{2}}\Big)\|\mathbb{F}-\mathbb{I}\|_{H^{2}}\nonumber \\
&\lesssim\|\nabla u\|_{H^{2}}\|\mathbb{F}-\mathbb{I}\|_{H^{2}}^{2}+\|\nabla u\|_{H^{2}}\|\mathbb{F}-\mathbb{I}\|_{H^{2}}.
\end{align}
Integrating (\ref{3.24}) from $0$ to $t$, we arrive at
\begin{align}\label{3.25}
\|(\mathbb{F}-\mathbb{I})(t)\|_{H^{2}}^{2}&\lesssim\sup_{0\leq \tau\leq t}\Big[\|(\mathbb{F}-\mathbb{I})(\tau)\|_{H^{2}}^{2}+\|(\mathbb{F}-\mathbb{I})(\tau)\|_{H^{2}}\Big]\int_{0}^{t}\|\nabla u(\tau)\|_{H^{2}}d\tau+\|\mathbb{F}_{0}-\mathbb{I}\|_{H^{2}}^{2}\nonumber \\
&\lesssim\sup_{0\leq \tau\leq t}\|(\mathbb{F}-\mathbb{I})(\tau)\|_{H^{2}}^{2}\int_{0}^{t}\|\nabla u(\tau)\|_{H^{2}}d\tau+\left(\int_{0}^{t}\|\nabla u(\tau)\|_{H^{2}}d\tau\right)^{2}+\|\mathbb{F}_{0}-\mathbb{I}\|_{H^{2}}^{2}.
\end{align}

\end{proof}

Next, we use $\mathbb{P}$ to $(\ref{3.7})_{1}, (\ref{3.7})_{2}$, we derive the integral representation for $(u,\mathcal{G})$ satisfying (\ref{3.7})
\begin{equation}\label{3.26}
\partial_{t}\begin{pmatrix}
u\\
\mathbb{P}\mathcal{G}
\end{pmatrix}=\begin{pmatrix}
-\Lambda^{2}&\Lambda\\
-\Lambda&0
\end{pmatrix}
\begin{pmatrix}
u\\
\mathbb{P}\mathcal{G}
\end{pmatrix}+\begin{pmatrix}
\mathbb{P}N_{1}\\
\mathbb{P}N_{2}
\end{pmatrix},
\end{equation}
where
\begin{align*}
&\mathbb{P}N_{1}:=-\mathbb{P}(u\cdot\nabla u)-\mathbb{P}\Big[g(\rho)(\Delta u+\Lambda\mathcal{G})\Big];\\
&\mathbb{P}N_{2}:=-\mathbb{P}\Lambda^{-1}\diver (u\cdot\nabla G)-\mathbb{P}\Lambda^{-1}\diver Q(\nabla u, G).
\end{align*}
Taking the Fourier transform of (\ref{3.26}), we have
\begin{equation}\label{3.27}
\partial_{t}\begin{pmatrix}
\widehat{u}\\
\widehat{\mathbb{P}\mathcal{G}}
\end{pmatrix}=\begin{pmatrix}
-|\xi|^{2}&|\xi|\\
-|\xi|&0
\end{pmatrix}
\begin{pmatrix}
\widehat{u}\\
\widehat{\mathbb{P}\mathcal{G}}
\end{pmatrix}+\begin{pmatrix}
\widehat{\mathbb{P}N_{1}}\\
\widehat{\mathbb{P}N_{2}}
\end{pmatrix}.
\end{equation}
The eigenvalues of the matrix $\begin{pmatrix}
-|\xi|^{2}&|\xi|\\
-|\xi|&0
\end{pmatrix}$ are given by
\begin{equation*}
\lambda_{1}=\frac{-|\xi|^{2}+\sqrt{\Gamma}}{2},~~\lambda_{2}=\frac{-|\xi|^{2}-\sqrt{\Gamma}}{2},
\end{equation*}
where $\Gamma=|\xi|^{4}-4|\xi|^{2}$.\\
The corresponding eigenvectors are
\begin{equation*}
\omega_{1}=\begin{pmatrix}
1\\
-\frac{|\xi|}{\lambda_{1}}
\end{pmatrix},~~\omega_{2}=\begin{pmatrix}
1\\
-\frac{|\xi|}{\lambda_{2}}
\end{pmatrix}.
\end{equation*}
Therefore, $e^{A t}$ can be explicitly written as
\begin{equation*}
e^{A t}=(\omega_{1},\omega_{2})\begin{pmatrix}
e^{\lambda_{1}t}&0\\
0&e^{\lambda_{2}t}
\end{pmatrix}(\omega_{1},\omega_{2})^{-1}=\begin{pmatrix}
e^{\lambda_{2}t}+\lambda_{1}G_{1}(t)&|\xi|G_{1}(t)\\
-|\xi|G_{1}(t)&e^{\lambda_{1}t}-\lambda_{1}G_{1}(t)
\end{pmatrix},
\end{equation*}
where $A:=\begin{pmatrix}
-|\xi|^{2}&|\xi|\\
-|\xi|&0
\end{pmatrix}, G_{1}(t):=\frac{e^{\lambda_{2}t}-e^{\lambda_{1}t}}{\lambda_{2}-\lambda_{1}}$.\\
Using Duhamel's principle, we have
\begin{align}\label{3.28}
&\widehat{u}(\xi,t)=\widehat{K}_{1}(\xi,t)\widehat{u}_{0}(\xi)+\widehat{K}_{2}(\xi,t)\widehat{\mathbb{P}\mathcal{G}}_{0}(\xi)\nonumber \\
&\ \ \ \ \ \ \ \ \ \ \ \ \ \ +\int_{0}^{t}\Big(\widehat{K}_{1}(\xi,t-\tau)\widehat{\mathbb{P}N_{1}}(\xi,\tau)+\widehat{K}_{2}(\xi,t-\tau)\widehat{\mathbb{P}N_{2}}(\xi,\tau)\Big)d\tau,\nonumber \\
&\widehat{\mathbb{P}\mathcal{G}}(\xi,t)=\widehat{K}_{3}(\xi,t)\widehat{u}_{0}(\xi)+\widehat{K}_{4}(\xi,t)\widehat{\mathbb{P}\mathcal{G}}_{0}(\xi)\nonumber \\
&\ \ \ \ \ \ \ \ \ \ \ \ \ \ +\int_{0}^{t}\Big(\widehat{K}_{3}(\xi,t-\tau)\widehat{\mathbb{P}N_{1}}(\xi,\tau)+\widehat{K}_{4}(\xi,t-\tau)\widehat{\mathbb{P}N_{2}}(\xi,\tau)\Big)d\tau,
\end{align}
where $\widehat{K}_{1}(\xi,t)=e^{\lambda_{2}t}+\lambda_{1}G_{1}(t),\widehat{K}_{2}(\xi,t)=|\xi|G_{1}(t), \widehat{K}_{3}(\xi,t)=-|\xi|G_{1}(t),\widehat{K}_{4}(\xi,t)=e^{\lambda_{1}t}-\lambda_{1}G_{1}(t)$.\\

The kernels $\widehat{K}_{1},\widehat{K}_{2},\widehat{K}_{3},\widehat{K}_{4}$ play a crucial role in the time-weighted energy estimates of $u$ and $\mathcal{G}$, the following proposition provide upper bounds by dividing the frequency space into subdomains.

\begin{Proposition}\label{p3.2}
The domain $\mathbb{R}^{3}$ is split into two subdomains, $\mathbb{R}^{3}=S_{1}\bigcup S_{2}$ with 
\begin{align*}
&S_{1}:=\left\{\xi\in\mathbb{R}^{3}:\Gamma\leq\frac{|\xi|^{4}}{4}~~ \text{or} ~~|\xi|^{2}\leq\frac{16}{3}\right\};\nonumber \\
&S_{2}:=\left\{\xi\in\mathbb{R}^{3}:\Gamma\geq\frac{|\xi|^{4}}{4}~~ \text{or} ~~|\xi|^{2}\geq\frac{16}{3}\right\}.
\end{align*}
(\textrm{I}) There exist two  constants $C>0$ and $c_{0}$ such that for any $\xi\in S_{1}$,
\begin{align*}
&Re\lambda_{1}\leq-\frac{|\xi|^{2}}{4}, ~~ -\frac{3|\xi|^{2}}{4}\leq Re\lambda_{2}\leq-\frac{|\xi|^{2}}{2};\nonumber \\
&|G_{1}(t)|\leq te^{-\frac{|\xi|^{2}}{4}t},~~|\widehat{K}_{i}(\xi,t)|\leq Ce^{-c_{0}|\xi|^{2}t}, i=1,2,3,4.
\end{align*}
(\textrm{II}) There exist two  constants $C>0$ and $c_{1}$ such that for any $\xi\in S_{2}$,
\begin{align*}
&\lambda_{1}\leq-2, ~~ -|\xi|^{2}\leq \lambda_{2}\leq-\frac{3|\xi|^{2}}{4};\nonumber \\
&|\widehat{K}_{i}(\xi,t)|\leq C|\xi|^{-1}e^{-c_{1}t}, i=2,3;\nonumber \\
&|\widehat{K}_{i}(\xi,t)|\leq Ce^{-c_{1}t}, i=1,4.
\end{align*}
\end{Proposition}
\begin{proof}
One can refer to Proposition 4.1 in \cite{Lin-Wu-Boardman2025} for a detailed calculations of a similar situation.
\end{proof}

\begin{Lemma}\label{le3.3}
Under the assumptions in Proposition \ref{p3.1}, then there exist a sufficiently small constant $\delta>0$ and some $\sigma\in(1,\sigma_{0}) (1<\sigma_{0}<\frac{3}{2})$ such that
\begin{align}\label{3.29}
&(1+t)^{\sigma}\|(u,\mathcal{G})(t)\|_{H^{2}}^{2}+\int_{0}^{t}(1+\tau)^{\sigma}\left(\|\Lambda u(\tau)\|_{H^{2}}^{2}+\|\Lambda \mathcal{G}(\tau)\|_{H^{1}}^{2}\right)d\tau\nonumber \\
&\lesssim (M_{0}^{2}+\|(\rho_{0},u_{0},\mathbb{F}_{0}-\mathbb{I})\|_{H^{2}}^{2}).
\end{align}
\end{Lemma}
\begin{proof}
Applying the integral representation (\ref{3.28}) and Plancherel’s theorem, we have
\begin{align}\label{3.30}
\|(u,\mathbb{P}\mathcal{G})(t)\|_{L^{2}(\mathbb{R}^{3})}&=\|(\widehat{u},\widehat{\mathbb{P}\mathcal{G}})(t)\|_{L^{2}(\mathbb{R}^{3})}\nonumber \\
&\leq J_{1}+J_{2}+J_{3}+J_{4}+J_{5}+J_{6},
\end{align}
where
\begin{align*}
&J_{1}=\|(\widehat{K}_{1}(t)\widehat{u}_{0},\widehat{K}_{3}(t)\widehat{u}_{0})\|_{L^{2}(\mathbb{R}^{3})};\\
&J_{2}=\|(\widehat{K}_{2}(t)\widehat{\mathbb{P}\mathcal{G}}_{0},\widehat{K}_{4}(t)\widehat{\mathbb{P}\mathcal{G}}_{0})\|_{L^{2}(\mathbb{R}^{3})};\\
&J_{3}=\int_{0}^{t}\|\widehat{K}_{1}(t-\tau)\widehat{\mathbb{P}N_{1}}(\tau)\|_{L^{2}(\mathbb{R}^{3})}d\tau;\\
&J_{4}=\int_{0}^{t}\|\widehat{K}_{2}(t-\tau)\widehat{\mathbb{P}N_{2}}(\tau)\|_{L^{2}(\mathbb{R}^{3})}d\tau;\\
&J_{5}=\int_{0}^{t}\|\widehat{K}_{3}(t-\tau)\widehat{\mathbb{P}N_{1}}(\tau)\|_{L^{2}(\mathbb{R}^{3})}d\tau;\\
&J_{6}=\int_{0}^{t}\|\widehat{K}_{4}(t-\tau)\widehat{\mathbb{P}N_{2}}(\tau)\|_{L^{2}(\mathbb{R}^{3})}d\tau.
\end{align*}
Next, we estimate each term on the right side of (\ref{3.30}) one by one. Using Lemma \ref{le2.3} and Proposition \ref{p3.2}, we have
\begin{align}\label{3.31}
J_{1}&\leq\|\widehat{K}_{1}(t)\widehat{u}_{0}\|_{L^{2}(S_{1})}+\|\widehat{K}_{1}(t)\widehat{u}_{0}\|_{L^{2}(S_{2})}+\|\widehat{K}_{3}(t)\widehat{u}_{0}\|_{L^{2}(S_{1})}
+\|\widehat{K}_{3}(t)\widehat{u}_{0}\|_{L^{2}(S_{2})}\nonumber \\
&\leq C\|e^{-c_{0}|\xi|^{2}t}\widehat{u}_{0}\|_{L^{2}(S_{1})}+C\|e^{-c_{1}t}\widehat{u}_{0}\|_{L^{2}(S_{2})}+C\||\xi|^{-1}e^{-c_{1}t}\widehat{u}_{0}\|_{L^{2}(S_{2})}\nonumber \\
&\leq C(1+t)^{-\frac{3}{4}}\|u_{0}\|_{L^{1}}+C(1+t)^{-\frac{3}{4}}\|u_{0}\|_{L^{2}},
\end{align}
where we have used some facts that $e^{-c_{1}t}(1+t)^{m}\leq C(c_{1},m)$ for any $m\geq0, t>0$ and $|\xi|\geq C$ for any $\xi\in S_{2}$. Similarly, we get 
\begin{align}\label{3.32}
J_{2}&\leq\|\widehat{K}_{2}(t)\widehat{\mathbb{P}\mathcal{G}}_{0}\|_{L^{2}(S_{1})}+\|\widehat{K}_{2}(t)\widehat{\mathbb{P}\mathcal{G}}_{0}\|_{L^{2}(S_{2})}
+\|\widehat{K}_{4}(t)\widehat{\mathbb{P}\mathcal{G}}_{0}\|_{L^{2}(S_{1})}+\|\widehat{K}_{4}(t)\widehat{\mathbb{P}\mathcal{G}}_{0}\|_{L^{2}(S_{2})}\nonumber \\
&\leq C\|e^{-c_{0}|\xi|^{2}t}\widehat{\mathbb{P}\mathcal{G}}_{0}\|_{L^{2}(S_{1})}+C\|e^{-c_{1}t}\widehat{\mathbb{P}\mathcal{G}}_{0}\|_{L^{2}(S_{2})}
+C\||\xi|^{-1}e^{-c_{1}t}\widehat{\mathbb{P}\mathcal{G}}_{0}\|_{L^{2}(S_{2})}\nonumber \\
&\leq C(1+t)^{-\frac{3}{4}}\|\mathcal{G}_{0}\|_{L^{1}}+C(1+t)^{-\frac{3}{4}}\|\mathcal{G}_{0}\|_{L^{2}}.
\end{align}
For $J_{3}$, using Proposition \ref{p3.2} and the projection operator $\mathbb{P}$ is bounded in $L^{2}$, one gets 
\begin{align}\label{3.33}
J_{3}&\leq C\int_{0}^{t}\|e^{-c_{0}|\xi|^{2}(t-\tau)}\widehat{N_{1}}(\tau)\|_{L^{2}(S_{1})}d\tau+C\int_{0}^{t}\|e^{-c_{1}(t-\tau)}\widehat{N_{1}}(\tau)\|_{L^{2}(S_{2})}d\tau\nonumber \\
&=C\int_{0}^{t-1}\|e^{-c_{0}|\xi|^{2}(t-\tau)}\widehat{N_{1}}(\tau)\|_{L^{2}(S_{1})}d\tau+C\int_{t-1}^{t}\|e^{-c_{0}|\xi|^{2}(t-\tau)}\widehat{N_{1}}(\tau)\|_{L^{2}(S_{1})}d\tau\nonumber \\
&\ \ \ \ +C\int_{0}^{t}\|e^{-c_{1}(t-\tau)}\widehat{N_{1}}(\tau)\|_{L^{2}(S_{2})}d\tau\nonumber \\
&\leq C\int_{0}^{t-1}\|e^{-c_{0}|\xi|^{2}(t-\tau)}\widehat{N_{1}}(\tau)\|_{L^{2}(\mathbb{R}^{3})}d\tau+C\int_{0}^{t}e^{-c_{1}(t-\tau)}\|N_{1}(\tau)\|_{L^{2}(\mathbb{R}^{3})}d\tau,
\end{align}
where, due to some facts that $e^{-c_{0}|\xi|^{2}(t-\tau)}\leq1$ for any $\tau\in[t-1,t]$ and $|\xi|\leq C$ for any $\xi\in S_{1}$, it is easy to see that
\begin{align*}
\int_{t-1}^{t}\|e^{-c_{0}|\xi|^{2}(t-\tau)}\widehat{N_{1}}(\tau)\|_{L^{2}(S_{1})}d\tau&\leq e^{c_{1}}\int_{t-1}^{t}e^{-c_{1}(t-\tau)}\|\widehat{N_{1}}(\tau)\|_{L^{2}(S_{1})}d\tau\\
&\leq C\int_{t-1}^{t}e^{-c_{1}(t-\tau)}\|\widehat{N_{1}}(\tau)\|_{L^{2}(\mathbb{R}^{3})}d\tau.
\end{align*}
Based on the above results, applying H\"{o}lder's inequality, Lemmas \ref{le2.1}-\ref{le2.4}, we have
\begin{align}\label{3.34}
J_{3}&\leq C\int_{0}^{t-1}(1+t-\tau)^{-\frac{3}{4}}\|u\cdot\nabla u+g(\rho)(\Delta u+\Lambda\mathcal{G})\|_{L^{1}}d\tau\nonumber \\
&\ \ \ \ +C\int_{0}^{t}(1+t-\tau)^{-\frac{3}{4}}\|u\cdot\nabla u+g(\rho)(\Delta u+\Lambda\mathcal{G})\|_{L^{2}}d\tau\nonumber \\
&\leq C\int_{0}^{t}(1+t-\tau)^{-\frac{3}{4}}(\|u\|_{L^{2}}\|\Lambda u\|_{L^{2}}+\|\rho\|_{L^{2}}\|\Lambda^{2} u\|_{L^{2}}+\|\rho\|_{L^{2}}\|\Lambda\mathcal{G})\|_{L^{2}})d\tau\nonumber \\
&\ \ \ \ +C\int_{0}^{t}(1+t-\tau)^{-\frac{3}{4}}(\|u\|_{L^{\infty}}\|\Lambda u\|_{L^{2}}+\|\rho\|_{L^{\infty}}\|\Lambda^{2} u\|_{L^{2}}+\|\rho\|_{L^{\infty}}\|\Lambda\mathcal{G})\|_{L^{2}})d\tau\nonumber \\
&\leq C\delta\int_{0}^{t}(1+t-\tau)^{-\frac{3}{4}}(\|\Lambda u\|_{H^{1}}+\|\Lambda\mathcal{G}\|_{L^{2}})d\tau.
\end{align}
For $J_{5}$, similar to (\ref{3.33}) and (\ref{3.34}), we get
\begin{align}\label{3.35}
J_{5}&\leq C\int_{0}^{t}\|e^{-c_{0}|\xi|^{2}(t-\tau)}\widehat{N_{1}}(\tau)\|_{L^{2}(S_{1})}d\tau+C\int_{0}^{t}\||\xi|^{-1}e^{-c_{1}(t-\tau)}\widehat{N_{1}}(\tau)\|_{L^{2}(S_{2})}d\tau\nonumber \\
&=C\int_{0}^{t-1}\|e^{-c_{0}|\xi|^{2}(t-\tau)}\widehat{N_{1}}(\tau)\|_{L^{2}(S_{1})}d\tau+C\int_{t-1}^{t}\|e^{-c_{0}|\xi|^{2}(t-\tau)}\widehat{N_{1}}(\tau)\|_{L^{2}(S_{1})}d\tau\nonumber \\
&\ \ \ \ +C\int_{0}^{t}\||\xi|^{-1}e^{-c_{1}(t-\tau)}\widehat{N_{1}}(\tau)\|_{L^{2}(S_{2})}d\tau\nonumber \\
&\leq C\delta\int_{0}^{t}(1+t-\tau)^{-\frac{3}{4}}(\|\Lambda u\|_{H^{1}}+\|\Lambda\mathcal{G}\|_{L^{2}})d\tau.
\end{align}
Similar to the estimates of $J_{3}$ and $J_{5}$, for $J_{4}$ and $J_{6}$, we have
\begin{align}\label{3.36}
J_{4}+J_{6}&\leq C\int_{0}^{t}\|e^{-c_{0}|\xi|^{2}(t-\tau)}\widehat{N_{2}}(\tau)\|_{L^{2}(S_{1})}d\tau+C\int_{0}^{t}\||\xi|^{-1}e^{-c_{1}(t-\tau)}\widehat{N_{2}}(\tau)\|_{L^{2}(S_{2})}d\tau\nonumber \\
&\ \ \ \ +C\int_{0}^{t}\|e^{-c_{1}(t-\tau)}\widehat{N_{2}}(\tau)\|_{L^{2}(S_{2})}d\tau\nonumber \\
&\leq C\int_{0}^{t}(1+t-\tau)^{-\frac{3}{4}}\|\Lambda^{-1}\diver (u\cdot\nabla G)-\Lambda^{-1}\diver Q(\nabla u, G)\|_{L^{1}}d\tau\nonumber \\
&\ \ \ \ +C\int_{0}^{t}(1+t-\tau)^{-\frac{3}{4}}\|\Lambda^{-1}\diver (u\cdot\nabla G)-\Lambda^{-1}\diver Q(\nabla u, G)\|_{L^{2}}d\tau\nonumber \\
&\leq C\delta\int_{0}^{t}(1+t-\tau)^{-\frac{3}{4}}(\|\Lambda u\|_{H^{1}}+\|\Lambda\mathcal{G}\|_{L^{2}})d\tau.
\end{align}
Combining the estimates of (\ref{3.31}), (\ref{3.32}), and (\ref{3.34})-(\ref{3.36}) together, and using Lemma \ref{le2.4}, we have 
\begin{align}\label{3.37}
\|(u,\mathcal{G})(t)\|_{L^{2}}&\leq C(1+t)^{-\frac{3}{4}}(\|(u_{0},\mathcal{G}_{0})\|_{L^{1}}+\|(u_{0},\mathcal{G}_{0})\|_{L^{2}})\nonumber \\
&\ \ \ \ +C\delta\int_{0}^{t}(1+t-\tau)^{-\frac{3}{4}}(\|\Lambda u\|_{H^{1}}+\|\Lambda\mathcal{G}\|_{L^{2}})d\tau\nonumber \\
&\leq C(1+t)^{-\frac{3}{4}}(M_{0}+\|(u_{0},\mathcal{G}_{0})\|_{L^{2}})\nonumber \\
&\ \ \ \ +C\delta\left(\int_{0}^{t}(1+t-\tau)^{-\frac{3}{2}}(1+\tau)^{-\frac{3}{2}}d\tau\right)^{\frac{1}{2}}\left(\int_{0}^{t}(1+\tau)^{\frac{3}{2}}(\|\Lambda u\|_{H^{1}}^{2}+\|\Lambda\mathcal{G}\|_{L^{2}}^{2})d\tau\right)^{\frac{1}{2}}\nonumber \\
&\leq C(1+t)^{-\frac{3}{4}}(M_{0}+\|(u_{0},\mathcal{G}_{0})\|_{L^{2}})\nonumber \\
&\ \ \ \ +C\delta(1+t)^{-\frac{3}{4}}\left(\int_{0}^{t}(1+\tau)^{\frac{3}{2}}(\|\Lambda u\|_{H^{1}}^{2}+\|\Lambda\mathcal{G}\|_{L^{2}}^{2})d\tau\right)^{\frac{1}{2}}.
\end{align}
Next, multiplying both sides of $(\ref{3.15})$ and $(\ref{3.20})$ by the time weight $(1+\tau)^{\frac{3}{2}}$, and integrating the results over $(0,t)$ respectively and considering (\ref{3.8}), we get
\begin{align}\label{3.38}
&(1+t)^{\frac{3}{2}}\|(u,\mathcal{G})(t)\|_{H^{2}}^{2}+\int_{0}^{t}(1+\tau)^{\frac{3}{2}}(\|\Lambda u\|_{H^{2}}^{2}+\|\Lambda\mathcal{G}\|_{H^{1}}^{2})d\tau\nonumber \\
&\leq C\|(\rho_{0},u_{0},\mathbb{F}_{0}-\mathbb{I})\|_{H^{2}}^{2}+\int_{0}^{t}(1+\tau)^{\frac{1}{2}}\|(u,\mathcal{G})(\tau)\|_{H^{2}}^{2}d\tau\nonumber \\
&\leq C\|(\rho_{0},u_{0},\mathbb{F}_{0}-\mathbb{I})\|_{H^{2}}^{2}+\int_{0}^{t}(1+\tau)^{\frac{1}{2}}\|(u,\mathcal{G})(\tau)\|_{L^{2}}^{2}d\tau\nonumber \\
&\ \ \ \ +\delta_{0}\int_{0}^{t}(1+\tau)^{\frac{3}{2}}\|(\Lambda u,\Lambda\mathcal{G})(\tau)\|_{H^{1}}^{2}d\tau+C_{\delta_{0}}\int_{0}^{t}\|(\Lambda u,\Lambda\mathcal{G})(\tau)\|_{H^{1}}^{2}d\tau.
\end{align}
Combining with (\ref{3.8}), (\ref{3.37}) and (\ref{3.38}), and choosing $\delta_{0}$ small enough, we have 
\begin{align}\label{3.39}
\|(u,\mathcal{G})(t)\|_{L^{2}}&\leq C(1+t)^{-\frac{3}{4}}(M_{0}+\|(u_{0},\mathcal{G}_{0})\|_{L^{2}})\nonumber \\
&\ \ \ \ +C\delta(1+t)^{-\frac{3}{4}}\left(\|(\rho_{0},u_{0},\mathbb{F}_{0}-\mathbb{I})\|_{H^{2}}^{2}+\int_{0}^{t}(1+\tau)^{\frac{1}{2}}\|(u,\mathcal{G})\|_{L^{2}}^{2}d\tau\right)^{\frac{1}{2}}.
\end{align}
Then we define the assistant energy $\mathcal{E}_{a}(t)$ as
\begin{equation}\label{3.40}
\mathcal{E}_{a}(t):=\sup_{0\leq t'\leq t}(1+t')^{\frac{3}{2}}\|(u,\mathcal{G})(t')\|_{H^{2}}^{2}.
\end{equation}
From $(\ref{3.8})$, $(\ref{3.39})$ and $(\ref{3.40})$,  we get
\begin{align}\label{3.41}
\mathcal{E}_{a}(t)&\leq C_{1}(M_{0}^{2}+\|(\rho_{0},u_{0},\mathbb{F}_{0}-\mathbb{I})\|_{H^{2}}^{2})+C_{0}\delta\int_{0}^{t}(1+\tau)^{-1}\mathcal{E}_{a}(\tau)d\tau,
\end{align}
where $C_{0}$ is determined by $(\ref{3.8})$, $(\ref{3.39})$ and $(\ref{3.40})$ and independent of $\delta$ and $\sigma$.\\
Using Gronwall's lemma for (\ref{3.41}), we obtain
\begin{align}\label{3.42}
\mathcal{E}_{a}(t)&\leq C_{1}(M_{0}^{2}+\|(\rho_{0},u_{0},\mathbb{F}_{0}-\mathbb{I})\|_{H^{2}}^{2})\exp\left(C_{0}\delta\int_{0}^{t}(1+\tau)^{-1}d\tau\right)\nonumber \\
&\leq C(M_{0}^{2}+\|(\rho_{0},u_{0},\mathbb{F}_{0}-\mathbb{I})\|_{H^{2}}^{2})(1+t)^{C_{0}\delta}.
\end{align}
Next, if $\sigma_{0}$ satisfies $1<\sigma<\sigma_{0}<\frac{3}{2}$, we can obtain
\begin{align}\label{3.43}
(1+t)^{\sigma_{0}}\|(u,\mathcal{G})(t)\|_{L^{2}}^{2}\leq\mathcal{E}_{a}(t)\leq C(M_{0}^{2}+\|(\rho_{0},u_{0},\mathbb{F}_{0}-\mathbb{I})\|_{H^{2}}^{2})(1+t)^{C_{0}\delta}.
\end{align}
From $(\ref{3.43})$, we get
\begin{align}\label{3.44}
\int_{0}^{t}(1+\tau)^{\sigma-1}\|(u,\mathcal{G})(\tau)\|_{L^{2}}^{2}d\tau&\leq C_{1}(M_{0}^{2}+\|(\rho_{0},u_{0},\mathbb{F}_{0}-\mathbb{I})\|_{H^{2}}^{2})\int_{0}^{t}(1+\tau)^{C_{0}\delta+\sigma-1-\sigma_{0}}d\tau\nonumber \\
&\leq C(M_{0}^{2}+\|(\rho_{0},u_{0},\mathbb{F}_{0}-\mathbb{I})\|_{H^{2}}^{2}),
\end{align}
where $C_{0}\delta<\sigma_{0}-\sigma$.\\
Similar to $(\ref{3.38})$, multiplying both sides of $(\ref{3.15})$ and $(\ref{3.20})$ by the time weight $(1+\tau)^{\sigma}$, and integrating the results over $(0,t)$ respectively and considering (\ref{3.8}), we finally derive that
\begin{align}\label{3.45}
&(1+t)^{\sigma}\|(u,\mathcal{G})(t)\|_{H^{2}}^{2}+\int_{0}^{t}(1+\tau)^{\sigma}\left(\|\Lambda u(\tau)\|_{H^{2}}^{2}+\|\Lambda \mathcal{G}(\tau)\|_{H^{1}}^{2}\right)d\tau\nonumber \\
&\lesssim\|(\rho_{0},u_{0},\mathbb{F}_{0}-\mathbb{I})\|_{H^{2}}^{2}+\int_{0}^{t}(1+\tau)^{\sigma-1}\|(u,\mathcal{G})(\tau)\|_{L^{2}}^{2}d\tau.
\end{align}
From $(\ref{3.44})$ and $(\ref{3.45})$, we immediately get (\ref{3.29}). This completes the proof of Lemma \ref{le3.3}.
\end{proof}

\begin{Remark}
According to Lemma \ref{le3.3}, the term $\int_{0}^{t}\|\nabla u(\tau)\|_{H^{2}}d\tau$ can be bounded. The details will be provided in the proof of next lemma, which aims to close the priori assumption in $(\ref{3.4})$.
\end{Remark}

\begin{Lemma}\label{le3.4}
Under the assumptions in Proposition \ref{p3.1}, then there exists a sufficiently small constant $\delta>0$ such that
\begin{equation}\label{3.46}
\|(\rho,u,\mathbb{F}-\mathbb{I})\|_{H^{2}}\lesssim\|(\rho_{0},u_{0},\mathbb{F}_{0}-\mathbb{I})\|_{H^{2}}+\|(\rho_{0},u_{0},\mathbb{F}_{0}-\mathbb{I})\|_{H^{2}}^{\frac{\sigma-1}{2\sigma}}M_{0}^{\frac{\sigma+1}{2\sigma}}.
\end{equation}
\end{Lemma}
\begin{proof}
Considering $(\ref{3.8})$ and $(\ref{3.29})$, making use of H\"{o}lder's inequality, we get
\begin{align}\label{3.47}
\int_{0}^{t}\|\Lambda u(\tau)\|_{H^{2}}d\tau&\lesssim \left(\int_{0}^{t}(1+\tau)^{-\frac{\sigma+1}{2}}d\tau\right)^{\frac{1}{2}}\left(\int_{0}^{t}(1+\tau)^{\sigma}\|\Lambda u(\tau)\|_{H^{2}}^{2}d\tau\right)^{\frac{\sigma+1}{4\sigma}}\left(\int_{0}^{t}\|\Lambda u(\tau)\|_{H^{2}}^{2}d\tau\right)^{\frac{\sigma-1}{4\sigma}}\nonumber \\
&\lesssim\|(\rho_{0},u_{0},\mathbb{F}_{0}-\mathbb{I})\|_{H^{2}}+\|(\rho_{0},u_{0},\mathbb{F}_{0}-\mathbb{I})\|_{H^{2}}^{\frac{\sigma-1}{2\sigma}}M_{0}^{\frac{\sigma+1}{2\sigma}}.
\end{align}
From Lemma \ref{le3.1} and Lemma \ref{le3.2}, $(\ref{0.7})$ and $(\ref{3.47})$, we arrive at
\begin{align}\label{3.48}
\|(\rho,\mathbb{F}-\mathbb{I})(t)\|_{H^{2}}^{2}&\lesssim\sup_{0\leq \tau\leq t}\|(\rho,\mathbb{F}-\mathbb{I})(\tau)\|_{H^{2}}^{2}\int_{0}^{t}\|\nabla u(\tau)\|_{H^{2}}d\tau+\left(\int_{0}^{t}\|\nabla u(\tau)\|_{H^{2}}d\tau\right)^{2}\nonumber \\
&\ \ \ \ +\|(\rho_{0},\mathbb{F}_{0}-\mathbb{I})\|_{H^{2}}^{2}\nonumber \\
&\lesssim\left(\delta+M_{0}^{\frac{\sigma+1}{2\sigma}}\delta^{\frac{\sigma-1}{2\sigma}}\right)\sup_{0\leq \tau\leq t}\|(\rho,\mathbb{F}-\mathbb{I})(\tau)\|_{H^{2}}^{2}+\|(\rho_{0},u_{0},\mathbb{F}_{0}-\mathbb{I})\|_{H^{2}}^{2}\nonumber \\
&\ \ \ \ +\|(\rho_{0},u_{0},\mathbb{F}_{0}-\mathbb{I})\|_{H^{2}}^{\frac{\sigma-1}{\sigma}}M_{0}^{\frac{\sigma+1}{\sigma}}.
\end{align}
Finally, we get (\ref{3.46}) by selecting $\delta$ small enough. 
\end{proof}

Combining the results of Lemma \ref{le3.3} and Lemma \ref{le3.4}, we can immediately obtain $(\ref{3.3})$. Thus, we have completed the proof of Proposition \ref{p3.1}.

\section*{Acknowledgements}

This work was partially supported by National Key R\&D Program of China (No. 2021YFA1002900), Guangzhou City Basic and Applied Basic Research Fund (No. 2024A04J6336), Yunnan Fundamental Research Projects (grant No. 202501AU070056).

\bigskip

{\bf Data Availability:} Data sharing is not applicable to this article.

\bigskip

{\bf Conflict of interest:} This work does not have any conflicts of interest.

\end{document}